\documentclass[12pt,twoside]{article}
\input{undertilde}
\usepackage{amsmath,amssymb,amstext,amsthm,enumerate}
\usepackage{latexsym}
\usepackage{graphics} 
\usepackage{epsfig} 
\usepackage{amsmath} 
\usepackage{amsfonts, amssymb} 
\usepackage{amsthm,enumerate}
\usepackage{color}
\usepackage{mathrsfs}
\usepackage{amscd}
\usepackage{tikz}
\usepackage{amsfonts,graphicx,latexsym}
\usepackage{xcolor}
\usepackage{amssymb, amsthm, amsmath, enumerate, graphicx, epsfig, hyperref,geometry,color}
\usepackage[T1]{fontenc}
\usepackage{amsfonts}
\usepackage{xcolor}
\usepackage{listings}
\usepackage{textcomp}
\usepackage{setspace}
\usepackage{palatino}
\usepackage{algorithm}
\usepackage{algorithmic}
\usepackage{graphicx}
\usepackage{multirow}
\usepackage{extarrows}
\usepackage{changes}
\usepackage{textcomp}
\allowdisplaybreaks[1]
\newcommand{\ssup}[1] {{{\scriptscriptstyle{({#1}})}}}

\def\1{{\mathchoice {1\mskip-4mu\mathrm l}      
{1\mskip-4mu\mathrm l}
{1\mskip-4.5mu\mathrm l} {1\mskip-5mu\mathrm l}}}

\DeclareFontFamily{U}{mathx}{\hyphenchar\font45}
\DeclareFontShape{U}{mathx}{m}{n}{
      <5> <6> <7> <8> <9> <10>
      <10.95> <12> <14.4> <17.28> <20.74> <24.88>
      mathx10
      }{}
\DeclareSymbolFont{mathx}{U}{mathx}{m}{n}
\DeclareFontSubstitution{U}{mathx}{m}{n}
\DeclareMathAccent{\widecheck}{0}{mathx}{"71}

\newcommand{\floor}[1]{\left\lfloor #1 \right\rfloor}
\newcommand{\N}     {\mathbb{N}}

\newcommand{\var}{\mbox{var}}
\newcommand{\cov}{\mbox{cov}}
\newcommand{\II}[1]{#1\!#1}

\hypersetup{
  colorlinks = true,
  linkcolor= red,
}
\def\bds{\begin{displaystyle}}
\def\eds{\end{displaystyle}}

\theoremstyle{definition}

\newtheorem{Prop}{Proposition}

\newtheorem{lemma}{Lemma}
\newtheorem{Theorem}{Theorem}

\newtheorem{Remark}{Remark}

\numberwithin{equation}{section}
\numberwithin{Theorem}{section}
\numberwithin{lemma}{section}
\numberwithin{Prop}{section}
\numberwithin{Remark}{section}
\numberwithin{corollary}{section}
\numberwithin{example}{section}
\numberwithin{D}{section}
\numberwithin{condition}{section}
\makeatletter

\newcommand{\Rmnum}[1]{\expandafter\@slowromancap\romannumeral #1@}
\makeatother
\begin{document}
\title{Invariance principle for biased Boostrap Random Walk.}
\author{A. Collevecchio, K. Hamza, Yunxuan Liu\\ \\School of Mathematical Science\\Monash University}

\date{}

\maketitle

\begin{abstract}
Our main goal is to study a class of processes whose increments are generated via a cellular automata rule.
Given the increments of a simple biased random walk, a new sequence of (dependent) Bernoulli random variables is produced. It is built, from the original sequence, according to a cellular automata rule. Equipped with these two sequences, we construct two more according to the same cellular automata rule. The construction is repeated a fixed number of times yielding an infinite array ($\{-K,\ldots,K\}\times\N$) of (dependent) Bernoulli random variables.
Taking partial sums of these sequences, we obtain a $(2K+1)$-dimensional process whose increments belong to the state space $\{-1,1\}^{2K+1}$.

The aim of the paper is to study the long term behaviour of this process. In particular, we establish transience/recurrence properties and prove an invariance principle.  The limiting behaviour of these processes depends strongly on the direction of the iteration, and exhibits few surprising features.
This work is motivated by an earlier investigation (see \cite{CHS15}), in which the starting sequence is symmetric, and by the related work \cite{Ferrari}.
\end{abstract}

\section{Introduction}
Our main goal is to study a class of processes whose increments are generated via cellular automata rule.
We consider a generalization of a model introduced in \cite{CHS15} which can be described as follows.
Consider a sequence of i.i.d. random variables taking values in $\{-1, 1\}$,  with mean $2p -1$, where $p \in (0,1)$. We place these values in a row of an infinite array. We label this row, row 0. The other rows are built recursively as follows. The $n$-th element of row $k$ is the partial product of the first $n$ elements of row $k-1$.
These ``downward'' iterations create an array that obeys the cellular automata rule that the $(k, n)$ entry is simply the product of the entries in $(k-1, n)$ and $(k, n-1)$, for $k\geq1, n\geq2$. The rule can be reversed to create, by ``upward'' iteration, rows of negative order.
This model is also akin to the three-dot system introduced by Ledrappier \cite{Led78} in that both satisfy the same three-dot identity
$$x_n(k)x_n(k-1)x_{n-1}(k)=1.$$
However the Bootstrap Random Walk (BRW) is restricted by design to the half-lattice as we impose the boundary condition that for any $k\in\mathbb{Z}$, $x_1(k)=x_1(0)$.

We emphasize that the randomness in this model is present only through row 0. In other words, given this row, the array is simply deterministic. In this regard, the result obtained in \cite{CHS15}, studying the case $p = 1/2$,  was quite striking, and  can be summarized as follows.  When $p = 1/2$, if we consider the  elements of the row $k$, with fixed $k\in\mathbb{Z}$, they are i.i.d. random variables taking values in $\{-1, 1\}$ with mean 0. Let $S_n^{\ssup k}$ be the sum of the first $n$ elements of row $k$. The process $\big(S_n^{\ssup k}\big)_n$, for each $k\in\mathbb{Z}$, is a simple symmetric random walk. Hence, properly re-scaled, it obeys a functional central limit theorem, i.e. weakly converges to Brownian motion. What could be surprising is that for any $K$, the process $S_n=(S_n^{\ssup{-K}},\ldots,S_n^{\ssup K})$, again properly re-scaled, converges to a $(2K+1)$-dimensional standard Brownian motion\footnote{We call standard Brownian motion, a multidimensional process whose components are independent Brownian motions.}.
In other words, the strong dependence  vanishes in the limit.

In this paper we generalize the above result to any $p\in(0,1)$. Here $S_n^{\ssup 1}$ is no longer a random walk and the question of whether or not a properly re-scaled $S_n^{\ssup 1}$ converges to a Brownian motion requires attention. It is remarkable that not only do we respond in the affirmative to this question but we do so collectively for $S_n^{\ssup{-K}},\ldots,S_n^{\ssup K}$. In fact, we reveal a dichotomy between the upward and downward processes. For the latter, we prove that the strong dependence between the components of $(S_n^{\ssup 0},\ldots,S_n^{\ssup K})$ still vanishes in the limit. In other words, even in this general setting a properly re-scaled $(S_n^{\ssup 0},\ldots,S_n^{\ssup K})$ converges to a $(K+1)$-dimensional standard Brownian motion. On the other hand, a limiting upward process retains some dependence beween the components of $(S_n^{\ssup{-K}},\ldots,S_n^{\ssup 0})$.  A properly re-scaled $(S_n^{\ssup{-K}},\ldots,S_n^{\ssup 0})$ converges to a correlated $(K+1)$-dimensional Brownian motion.

Furthermore, in the general setting of $p\neq1/2$, the recurrence analysis of $S_n$ is even more interesting as the drift, present in $(S_n^{\ssup 0})_n$, decays exponentially after only the first downward iteration.
In fact, although $(S_n^{\ssup 0})_n$ is a random walk with drift $2p -1$, the process $(S_n^{\ssup k})_n$ with  $k\geq1$, visits the origin infinitely often. However, this is not the case for the upward process whose components, $(S_n^{\ssup k})_n$, $k\leq-1$, are all transient.

In the original model introduced in \cite{CHS15}, the increments of $(S_n^{\ssup k})_n$ are allowed to take values in a finite set $\mathcal{U} = \{u_0, u_1, . . . , u_{q-1}\}\subset\mathbb{R}$ equipped with an operation $\bigotimes$ such that $(\mathcal{U},\bigotimes)$ is an Abelian group.  Here $q$ is assumed to be a prime number. The same generalisation can be undertaken here. However, many of the analytic expressions we rely on become significantly more cumbersome and significantly less intuitive. For this reason, we choose to present the results in the special case $\mathcal{U}=\{-1,1\}$ equipped with the usual multiplication. Different features of a similar model were studied in \cite{Ferrari}.

The paper is structured as follows. In Section \ref{Def:S}, we introduce the model, state the main results, give an account of the existing literature and describe the importance of our results within this context.
Section 3 sets the model within a Markov chain setting. It concludes with a study of the recurrence properties of $(S_n)_n$. Section 4 proves the invariance principle giving an explicit representation for the normalizing constants. The paper concludes with an Appendix that contains three combinatorial lemmas.

\section{Model Set-up and main results.}\label{Def:S}

\subsection{The model}

Denote by $\N$ the set of positive integers, and  set $\N_0 = \N \cup \{0\}$.
Consider the following model. Let $(\xi_i)_i$ be a sequence of independent random variables having the Bernoulli$(p)$ distribution on $\{-1,1\}$:
\begin{equation}\label{Bernoulli}
\rho(x) = p^{\frac{1+x}2}(1-p)^{\frac{1-x}2},
\end{equation}
where $p \in (0,1)$. We build an array $(\eta_{k,n})_{k\in\mathbb{Z},n\in\mathbb{N}}$ as follows. For $k\in\N_0,n\in\N$, we set
$$\eta_{0, n} =  \xi_n,\qquad\eta_{k+1,n} = \prod_{m=1}^n \eta_{k,m}\qquad\mbox{and}\qquad\eta_{-k,n}=\eta_{-k+1,n-1}\eta_{-k+1,n}.$$
\begin{figure}[h]
\centering
\includegraphics[width=\textwidth]{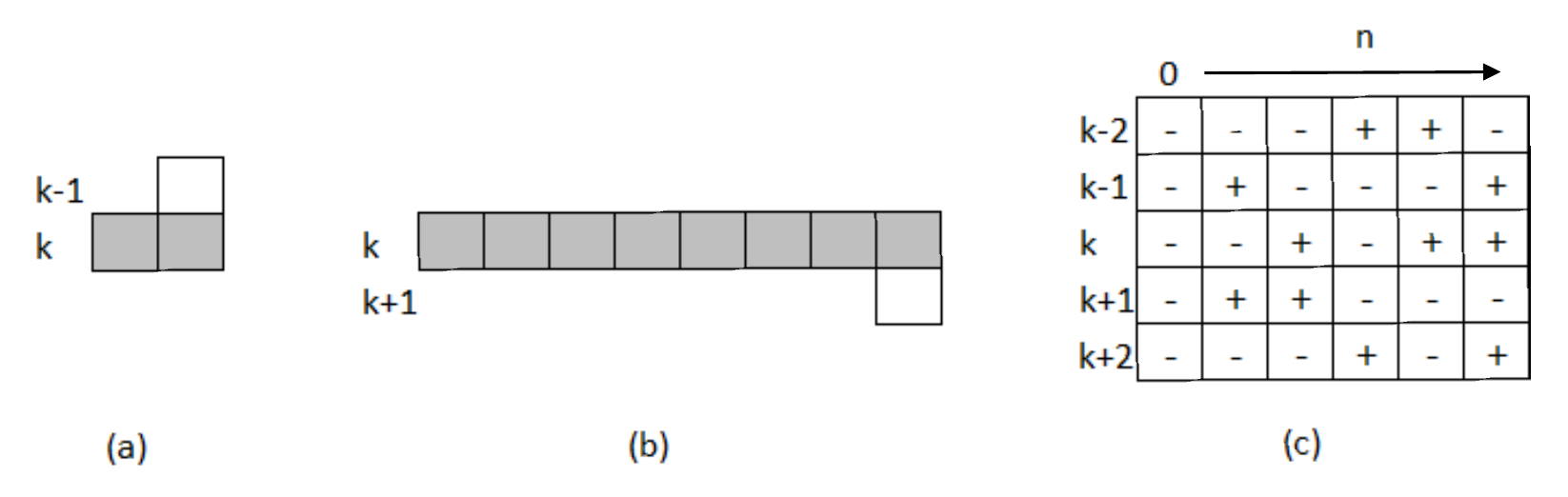}
\caption{(a) Upward process: the white cell is the product of the two grey cells immediately below and to the left of it.
(b) Downward process: the white cell is the product of all grey cells immediately above and to the left of it.
(c) $\eta_{k,n}$ array: ``-'' for $-1$ and ``+'' for $1$.}\label{3dot0}
\end{figure}
These rules are represented graphically in Figure \ref{3dot0}.
As seen in \cite{CHS15}, $\eta_{k,n}$ can then be written in terms of the original sequence $(\xi_n)_n$:
$$\eta_{k,n} = \prod_{m=1}^n\xi_{n-m+1}^{\nu_{k,m}},$$
where $(\nu_{k,n})_{k\in\mathbb{Z},n\in\mathbb{N}}$ is defined as the array such that, for any $k\in\mathbb{Z},n\in\mathbb{N}$,
$$\nu_{0,1}=1,\quad\nu_{0,n+1}=0,\quad\nu_{k,1}=1,\quad\nu_{k,n}= {n+k-2\choose n-1}=\frac{(n+k-2)\ldots k}{(n-1)!} \mod 2.$$

As already mentioned, unlike the symmetric case for which all rows play identical roles, in the non-symmetric setting ($p\neq1/2$), which we focus on in this paper, row 0 plays a unique role and the behaviours of the upward and downward processes are very different, in many ways, contrasted.

\subsection{The main results}

The main purpose of the paper is to describe the long term behaviour of the $(2K+1)$-dimensional process $S_n = (S_n^{\ssup{-K}},\ldots,S_n^{\ssup K})$, where $K \in \N$ is fixed throughout the paper.

We start with a recurrence analysis of the process $S_n$.

\begin{Theorem}\label{recS}
\begin{enumerate}
\item $d=1$. The process $(S_n^{\ssup k})_n$ visits 0 finitely many times, for any $k\in\II[-K,0\II]$\footnote{We write $\II[k,\ell\II]$ for the set $[k,\ell]\cap\mathbb{Z}$.}, and infinitely often, for any $k\in\II[1,K\II]$.
\item $d=2$. For any $k,\ell\in\II[1,K\II]$, the 2-dimensional process $\big((S_n^{\ssup k},S_n^{\ssup\ell})\big)_n$ visits the origin $(0,0)$ infinitely often.
\item $d\geq3$ ($d\leq K$). Any $d$-dimensional process projection of $(S_n^{\ssup 1},\ldots,S_n^{\ssup K})$ visits the origin finitely many times.
\end{enumerate}
\end{Theorem}

Next, we state an invariance principle for $S_n$. To this end we introduce the following notations.

For $k,\ell\geq0$ and $n\in\N$, $\beta_{k,\ell} = {k\choose\ell}\mod 2$, $\theta(k)$ is the number of 1's in the binary representation of $k$, $2^{\kappa(n)}$ is the largest power of 2 that divides $n$, $B(k,\ell) = \#\{j\in\II[1,k\vee\ell\II]:\ \beta_{k,j}+\beta_{\ell,j}=1\}$\footnote{We use the convention that ${n\choose m}=0$ whenever $n<m$.}, $\omega_{k,n}= \#\{m\in\II[1,n\II]:\nu_{k,m}=1\}$ and
$$\mathcal{B}(k,\ell,n) = 2^{\theta(k)} + 2^{\theta(\ell)} - 2\sum_{j=0}^{k\wedge(-n+\ell)}\beta_{k,j}\beta_{\ell,n+j}.$$
As usual we adopt the convention that, for $m<n$, $\sum_{i=n}^mx(i)=0$ and $\prod_{i=n}^mx(i)=1$.

We establish in Lemma \ref{B} the identity
$$B(k,\ell) = 2^{\theta(k)}+2^{\theta(\ell)}-2\sum_{j=0}^{k\wedge\ell}\beta_{k,j}\beta_{\ell,j} = \mathcal{B}(k,\ell,0).$$

Let
$$\mu(k) = \left\{
\begin{array}{ll}
(2p-1)^{2^{\theta(-k)}} & k\leq0\\
0 & k\geq1
\end{array}
\right.$$
and
$$\Sigma(k,\ell) = \left\{
\begin{array}{ll}
\sum_{n=1}^{-\ell}(2p-1)^{\mathcal{B}(-k,-\ell,n)} + \sum_{n=1}^{-k}(2p-1)^{\mathcal{B}(-\ell,-k,n)} & \\
\quad +\ (2p-1)^{B(-k,-\ell)} - (-k-\ell+1)(2p-1)^{2^{\theta(-k)}+2^{\theta(-\ell)}} & k,\ell\leq0\\
1 + 2\sum^\infty_{n=1,\ k\leq2^{\kappa(n)}}(2p-1)^{\omega_{k,n}} & k=\ell\geq1\\
0 & \mbox{otherwise}
\end{array}
\right.$$

For any $k\in\mathbb{Z}$, $n \in \N$ and $t \in [0,1]$, define
$$S_n^{\ssup k}(t) = \sum_{m=1}^{\floor {tn}}\eta_{k,m},\quad\bar{S}^{\ssup k}_n(t):=\big(S^{\ssup k}_n(t)-\mu(k)\lfloor nt \rfloor\big)+\left(nt-\lfloor nt \rfloor\right)\big(\eta_{k,\lfloor nt \rfloor+1}-\mu(k)\big)$$
and
$$U_n(t) = \frac1{\sqrt{n}}\left(\overline{S}^{\ssup{-K}}_n(t),\ldots,\overline{S}^{\ssup K}_n(t)\right),$$
where $\lfloor x \rfloor$ is the integer part of $x$. We shall also write $\widecheck{U}_n(t)$ for the projection of $U_n(t)$ on the components $(0,\ldots,K)$:
$$\widecheck{U}_n(t) = \frac1{\sqrt{n}}\left(\overline{S}^{\ssup 0}_n(t),\ldots,\overline{S}^{\ssup K}_n(t)\right).$$

\begin{Theorem}\label{TH2}
$(U_n(t))_{t \in [0,1]}$ converges weakly to a $(2K+1)$-dimensional Brownian motion whose covariance matrix is given by the matrix $\Sigma(k,\ell)$. In particular, $(\widecheck{U}_n(t))_{t \in [0,1]}$ converges weakly to a $(K+1)$-dimensional standard Brownian motion.
\end{Theorem}

\subsection{Literature Review} \label{LitRev}

\subsubsection{Elephant Random Walk}
The class of processes that we are studying in this paper were introduced and studied in \cite{CHS15}. As we mentioned before,  each of the $(S^{\ssup k}_n)_n$  with $k \ge 1$,  is an example of process with long memory. The  process $S_n^{\ssup 1}$ has many similarities, in terms of construction, with the Elephant Walk, which can be described as follows. Elephant random walks captured a lot of attention in the literature in recent years, as they are long-memory processes connected with random recursive trees, and exhibit anomalous diffusion properties. This family of processes have a parameter $p$ and are built as follows. The first row of the array is built in the same way the BRW is assembled, i.e. $\mathcal{E}_{1, i} = \xi_i$ where the $\xi$ are i.i.d. in $\{-1, 1\}$ with mean $2p -1$.
Let $(U_k)_k$ be a sequence  of independent  random variables, where $U_k$ is discrete uniform over the set $\II[1,k\II]$. Fix $\mathcal{E}_{2, 1} = \mathcal{E}_{1, 1}$ and suppose defined $\mathcal{E}_{2, n}$ with $n \in \N$. We define $\mathcal{E}_{2, n+1}$ as $\xi_{n+1}\mathcal{E}_{2, U_n}$. Define $E_n = \sum_{j=1}^n \mathcal{E}_{2, j}$. The process $(E_n)_n$ is known in the literature as an Elephant Random Walk and was introduced  by Sch\"utz and Trimper \cite{ST04}. It was studied, among others,  by Baur and Bertoin \cite{BB16} where central limit theorems are proved for a certain range of values of $p$, and a convergence to non-Gaussian limits for other values. See also Gonzales, Gonz\'alez-Navarrete and Lambert \cite{Go16}.

\subsubsection{Dependent percolation}
Here we focus on the array $(\eta_{\ell,n})_{\ell,n}$. If we color blue the entries that are equal to $-1$ and orange the ones that equal $+1$ we observe a typical structure of cellular automata. The orange sizes are organized as triangles, we call islands, while the blue cells percolate in the manner of a river until they either hit an orange triangle or reach the extremities of a finite array. When we change the values of $p$ the triangles have the tendency of becoming bigger as $p$ approaches 1 or 0.  Theorem \ref{TH2} implies that the difference between the number of $+1$ and $-1$  in any horizontal strip which does not include the first row, and has $K\times n$ sites, re-scales as $\sqrt n$ and converges to a normal distribution with mean 0. Percolation properties of the unrestricted (on the entire plane) three-dot system are studied by Quint \cite{Qui08}.

\section{A Markov chain set-up}
Consider the state space $E=\{-1,1\}^{2K+1}$ on which we define a process $X$ as follows.
For each vector $e \in E$, we denote by $e(-K),\ldots,e(-1),e(0),e(1),\ldots,e(K)$ its coordinates, and by $e^+$ and $e^-$ be the two ``forward'' vectors satisfying, for $k\in\II[1,K\II]$,
\begin{equation}\label{efwd}
e^\pm(0) = \pm1,\quad e^{\pm}(k) = e^{\pm}(k-1)e(k)\quad\mbox{and}\quad e^{\pm}(-k)  = e^{\pm}(-k+1)e(-k+1).
\end{equation}
We shall also need the ``backward'' vectors $^+e$ and $^-e$ satisfying
\begin{equation}\label{ebwd}
^{\pm}e(-K) = \pm1\quad\mbox{and}\quad ^{\pm}e(k)  = e(k-1)e(k),
\end{equation}
$\ k\in\II[-K+1,K\II]$. Then the transition probabilities defined, for any $n \in \N$ and any $e \in E$, by
\begin{equation}\label{1.0.1}
\mathbb{P}(X_n=e^\pm| X_{n-1}=e) = \rho(\pm1),
\end{equation}
specify a time homogeneous Markov Chain $X=(X_n)_n$ on $E$. We shall denote its transition probability matrix $P$ and, as usual, index the probability measure $\mathbb{P}$ and the corresponding expectation $\mathbb{E}$ with its initial distribution.

We observe that $X_n(0)$ is independent of $\sigma(X_0,\ldots,X_{n-1})$ and has the Bernoulli distribution $\rho$. Indeed,
$$\mathbb{P}(X_n(0)=\pm1|X_{n-1}=e) = \mathbb{P}(X_n=e^\pm| X_{n-1}=e) = \rho(\pm1).$$
This enables us to see the process $(\eta_n)_n$ as the Markov chain $(X_n)_n$ started at $\mathbf{1}$, and the sequence $(\xi_n)_n$ as the innovation $(X_n(0))_n$ in the chain $(X_n)_n$.

In fact, for any integers $k\in\II[0,K\II]$ and $n\in\N$,
\begin{eqnarray}
X_n(-k) & = & X_{n-1}(-k+1)X_n(-k+1)\ =\ \prod_{j=0}^{k\wedge n}X_{n-j}(k\wedge n-k)^{\beta_{k\wedge n,j}}\label{xirep}\\
X_n(k) & = & X_0(k) \prod_{m=1}^nX_m(k-1)\ =\ \left(\prod_{\ell=0}^{k-1}X_0(k-\ell)^{\alpha_{\ell,n}}\right)\left(\prod_{m=1}^nX_{n-m+1}(0)^{\nu_{k,m}}\right)\nonumber
\end{eqnarray}
where $\bds\nu_{k,n} = {n+k-2 \choose k-1} \mod 2\eds$, $\alpha_{k,n} = \nu_{k+1,n}$ and $\bds\beta_{k,\ell} = {k \choose \ell} \mod 2\eds$.\footnote{$\alpha_{k,n}=\left(\!{n\choose k}\!\right)$ is the number of ways to choose a multiset of cardinality $k$ from a set of cardinality $n$.}

In particular, $X_n(-k)$ is a function of $(X_0(-k),\ldots,X_0(0),\ldots,X_n(0))$. The dependence can be elaborated on depending on the relative values of $k$ and $n$. If $k<n$, then the dependence is only on $X_{n-k}(0),\ldots,X_n(0)$. While, if $k\geq n$, then the dependence is on $(X_0(-k+n),\ldots,X_0(0),X_1(0),\ldots,X_n(0))$. Similarly, $X_n(k)$ is a function of $(X_0(k),\ldots,X_0(1),X_1(0),\ldots,X_n(0))$, but not of $X_0(0)$.

Two other processes are of interest and will be discussed in the sequel. They are the two projections on the coordinates $(-K,\ldots,0)$ and the coordinates $(0,\ldots,K)$. We call the former $(\widehat{X}_n)_n$ and the latter $(\widecheck{X}_n)_n$, and retain the original numbering of the components. In other words, the components of $\widehat{X}_n$ are numbered $-K,\ldots,0$, and those of $\widecheck{X}_n$ as $0,\ldots,K$. In all, for $k\in\II[0,K\II]$,
$$\widehat{X}_n(-k) = X_n(-k)\quad\mbox{and}\quad \widecheck{X}_n(k) = X_n(k).$$
It is easy to see that $(\widehat{X}_n)_n$ is a Markov chain with the following transition probabilities. Let, for $\widehat{e}\in E'=\{-1,1\}^{K+1}$, $\widehat{e}^\pm$ be the projection of $e^\pm$ (see \eqref{efwd}) on the coordinates $(-K,\ldots,0)$. Then
$$\mathbb{P}(\widehat{X}_n=\widehat{e}^\pm| \widehat{X}_{n-1}=\widehat{e}) = \rho(\pm1).$$
Similarly, $(\widecheck{X}_n)_n$ is a Markov chain with transition probabilities:
$$\mathbb{P}(\widecheck{X}_n=\widecheck{e}^\pm| \widecheck{X}_{n-1}=\widecheck{e}) = \rho(\pm1),$$
where $\widecheck{e}\in E'$ and $\widecheck{e}^\pm$ is the projection of $e^\pm$ on the coordinates $(0,\ldots,K)$.

\begin{Prop}\label{irreducible}
For any $e_0,e_{2K+1}\in E$ there exists a unique chain $(e_1,\ldots,e_{2K})$ such that
\begin{eqnarray*}
\mathbb{P}(X_{2K+1}=e_{2K+1}|X_0=e_0) & = & \mathbb{P}(X_1=e_1,\ldots,X_{2K+1}=e_{2K+1}|X_0=e_0)\\
& = & \prod_{k=1}^{2K+1}\rho(e_k(0)).
\end{eqnarray*}
And for any $\widehat{e}_{K+1}\in E'$, there exists a unique sequence $(x_1,\ldots,x_K)$ such that for any $\widehat{e}_0\in E'$,
$$\mathbb{P}(\widehat{X}_{K+1}=\widehat{e}_{K+1}|X_0=\widehat{e}_0) = \rho(\widehat{e}_{K+1}(0))\prod_{k=1}^K\rho(x_k).$$
In other words, $(\widehat{X}_n)_n$ reaches stationarity after $K+1$ steps.
\end{Prop}
\begin{proof}
The proof relies on Figure \ref{3dot}.

\begin{figure}[h]
\centering
\includegraphics[width=.75\textwidth]{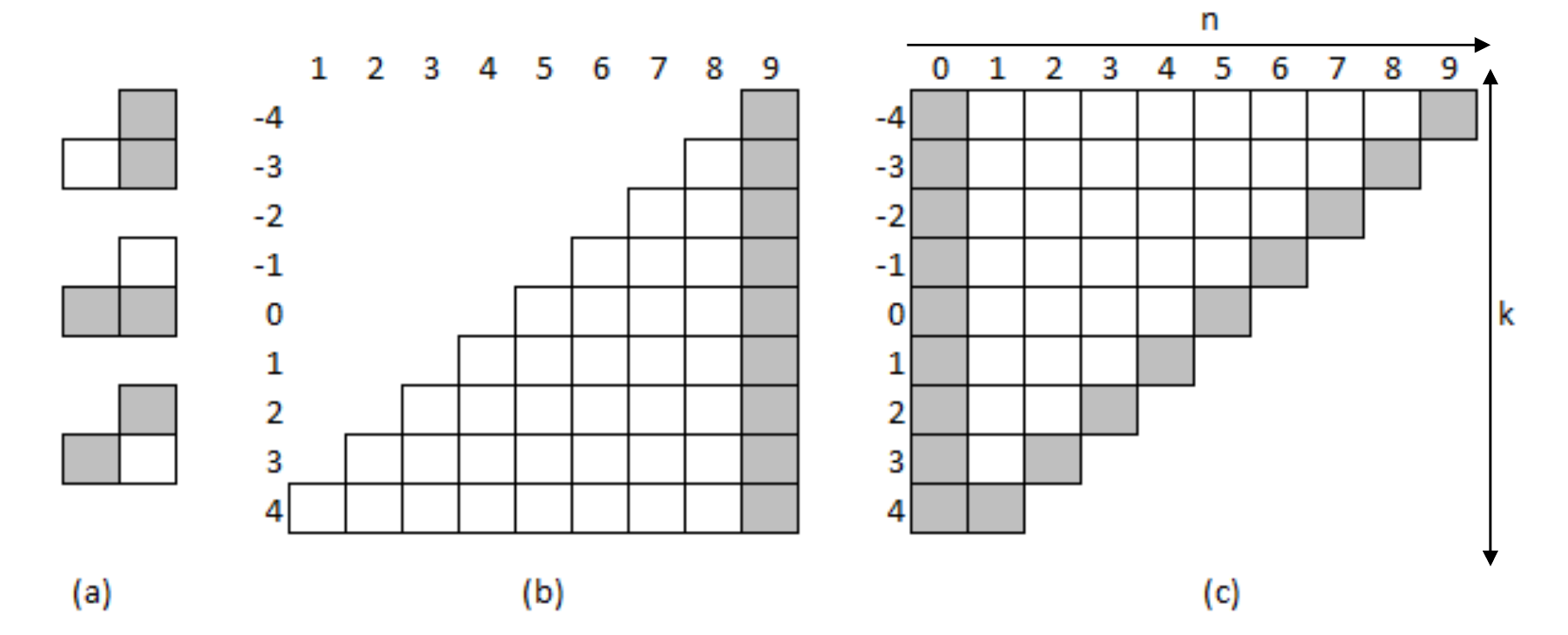}{\centering}
\caption{The three-dot system ($K=4$). The white cells are uniquely determined from the knowledge of the grey cells.}\label{3dot}
\end{figure}

(a) describes the three-point system:
\begin{equation}
X_n(i) X_n(i-1) X_{n-1}(i) =1
\end{equation}
and illustrates the fact that setting any two cells fixes the value of the third.

Using (a), (b) shows how, for a given value of $e_{2K+1}$, we can uniquely complete the half-array of size $(2K+1)\times(2K+1)$.

Again using (a), (c) shows that, given $e_0$ and $(e_1(K),e_2(K-1),\ldots,e_{2K+1}(-K))$, the remaining cells, and in particular those on row 0, are uniquely determined. Therefore, one can find a unique sequence $(e_1,\ldots,e_{2K})$ such that
\begin{eqnarray*}
\mathbb{P}(X_0=e_0,X_{2K+1}=e_{2K+1}) & = & \mathbb{P}(X_0=e_0,X_1(0)=e_1(0),\ldots,X_{2K+1}(0)=e_{2K+1}(0))\\
& = & \left(\prod_{k=1}^{2K+1}\rho(e_k(0))\right)\mathbb{P}(X_0=e_0).
\end{eqnarray*}
The second part of the statement is shown in a similar way. Note that $\mathbb{P}(\widehat{X}_{K+1}=\widehat{e}_{K+1}|X_0=\widehat{e}_0)$ does not depend on $\widehat{e}_0$.
\end{proof}

\begin{Remark}
From \cite{CHS15}, we know that, if $p=\frac{1}{2}$, then the Markov chain reaches its stationary distribution after exactly $2K+1$ steps regardless of the initial state. In other words, when $p = 1/2$, $(X_n)_n$ is a $(2K+1)$-dependent process that mixes perfectly after $2K+1$ steps.

If $p\neq\frac{1}{2}$, this is no longer the case. In fact, while the knowledge of $X_{2K+1}$ is equivalent to that of $X_1(K),\ldots,X_{2K+1}(K)$ and does not involve $X_0$, to recover the vectors $X_1,\ldots,X_{2K}$ in their entirety, one needs $X_0$ (as well as $X_{2K+1}$). Therefore, the sequence $e_1(0),\ldots,e_{2K+1}(0)$ in Proposition \ref{irreducible}, will in general depend on $e_0$.

In fact, while $(\widecheck{X}_n)_n$ fails to mix perfectly after a fixed number of steps, $(\widehat{X}_n)_n$ reaches stationarity after precisely $K+1$ steps. This highlights again the difference between the case $p=1/2$, where downward and upward processes are indistinguishable, and the case $p\neq1/2$ for which the same two processes exhibit very different behaviours.
\end{Remark}

\begin{Prop}\label{prop1.1}
The Markov chains $(X_n)_n$, $(\widehat{X}_n)_n$ and $(\widecheck{X}_n)_n$ are irreducible and aperiodic.

Moreover, their unique stationary distributions $\pi$, $\widehat{\pi}$ and $\widecheck{\pi}$ are: for any $e\in E$ and $\widehat{e},\widecheck{e}\in E'$,
\begin{eqnarray}
\pi(e) & = & 2^{-K}\prod_{k=0}^K\rho\left(\prod_{\ell=0}^ke(-\ell)^{\beta_{k,\ell}}\right)\label{eq:pi}\\
\widehat{\pi}(\widehat{e}) & = & \prod_{k=0}^K\rho\left(\prod_{\ell=0}^k\widehat{e}(-\ell)^{\beta_{k,\ell}}\right)\label{eq:pi-up}\\
\widecheck{\pi}(\widecheck{e}) & = & 2^{-K}\rho\left(\widecheck{e}(0)\right)\label{eq:pi-down}
\end{eqnarray}
\end{Prop}
\begin{proof}
Irreducibility is a direct consequence of Proposition \ref{irreducible}. Aperiodicity now follows from the fact that if $e=\mathbf{1}$, then $e^+ = e$ and $P(e,e)=p>0$. Similarly for $\widehat{\pi}$ and $\widecheck{\pi}$.

Next we show that if $X_0$ has distribution $\pi$ as given in \eqref{eq:pi}, then so does $X_1$. The same reasoning extends to $(\widehat{X}_n)_n$ and $(\widecheck{X}_n)_n$.

Let $e\in E$. The two vectors $^-e$ and $^+e$ are the only vectors in $E$ such that $P(^{\pm}e,e)>0$.
Now, for $k<K$,
\begin{eqnarray*}
\prod_{\ell=0}^k{^{\pm}e}(-\ell)^{\beta_{k,\ell}} & = & \prod_{\ell=0}^k\big(e(-\ell)e(-\ell-1)\big)^{\beta_{k,\ell}}\ =\ \left(\prod_{\ell=0}^ke(-\ell)^{\beta_{k,\ell}}\right)\left(\prod_{\ell=1}^{k+1}e(-\ell)^{\beta_{k,\ell-1}}\right)\\
& = & \prod_{\ell=0}^{k+1}e(-\ell)^{\beta_{k+1,\ell}}
\end{eqnarray*}
and, for $k=K$,
$$\prod_{\ell=0}^K{^{\pm}e}(-\ell)^{\beta_{K,\ell}} = \pm\prod_{\ell=0}^{K-1}\big(e(-\ell)e(-\ell-1)\big)^{\beta_{K,\ell}},$$
so that
$$\rho\left(\prod_{\ell=0}^K{^-e}(-\ell)^{\beta_{K,\ell}}\right)+\rho\left(\prod_{\ell=0}^K{^+e}(-\ell)^{\beta_{K,\ell}}\right)=1.$$
It follows that
\begin{eqnarray*}
\mathbb{P}_\pi(X_1 = e) & = & \mathbb{P}_\pi\big(X_1 = e, X_0 = {^-e}\big) + \mathbb{P}_\pi\big(X_1 = e, X_0 = {^+e}\big)\\
& = & \mathbb{P}_\pi\big(X_0 = {^-e},X_1(0) = e(0)\big) + \mathbb{P}_\pi\big(X_0 = {^+e},X_1(0) = e(0)\big)\\
& = & \mathbb{P}_\pi\big(X_0 = {^-e}\big)\mathbb{P}_\pi\big(X_1(0) = e(0)\big) + \mathbb{P}_\pi\big(X_0 = {^+e}\big)\mathbb{P}_\pi\big(X_1(0) = e(0)\big)\\
& = & 2^{-K}\rho(e(0))\prod_{k=0}^{K-1}\rho\left(\prod_{\ell=0}^{k+1}e(-\ell)^{\beta_{k+1,\ell}}\right)\\
& = & 2^{-K}\rho(e(0))\prod_{k=1}^K\rho\left(\prod_{\ell=0}^ke(-\ell)^{\beta_{k,\ell}}\right)\ =\ \pi(e).
\end{eqnarray*}
\end{proof}

In the next proposition we present results on the marginal distributions of $\pi$.
Let $\pi_k$ be the $k$-marginal of $\pi$; i.e. the projection of $\pi$ on component $k$ ($k\in\II[-K,K\II]$):
$$\pi_k(x) = \sum_{\substack{e\in E\\ e(k)=x}}\pi(e).$$
The following proposition describes these marginal distributions and how the corresponding random variables relate to each other. It makes use of the celebrated Lucas Theorem \cite{Lucas}.

\begin{Theorem}[Lucas]
A binomial coefficient ${n\choose m}$ is divisible by a prime $p$ if and only if at least
one of the base $p$ digits of $m$ is greater than the corresponding digit of $n$.
\end{Theorem}

\begin{Prop}\label{pidist}
Let $\Xi$ be a random vector whose distribution is $\pi$.
Then
\begin{enumerate}
\item $\widehat{\Xi},\Xi(1),\ldots,\Xi(K)$ are independent and, $\widehat{\Xi}$ has distribution $\widehat{\pi}$, $\Xi(0)$ has distribution $\rho$ and $\Xi(1),\ldots,\Xi(K)$ are symmetric Bernoulli random variables;
\item for $1\leq k\leq K$ and $x_0,\ldots,x_{k-1}\in\{-1,1\}$,
$$\mathbb{P}\left(\Xi(-k)=1|\Xi(-k+1)=x_{k-1},\ldots,\Xi(0)=x_0\right) = \rho\left(\prod_{\ell=0}^{k-1}x_{\ell}^{\beta_{k,\ell}}\right);$$
\item for $0\leq k\leq K$,
$$\mathbb{E}\big[\Xi(-k)\big] = (2p-1)^{2^{\theta(k)}} = \mu(-k);$$
\item for $0\leq k\leq\ell\leq K$,
$$\mathbb{E}\big[\Xi(-k)\Xi(-\ell)\big] = (2p-1)^{B(k,\ell)},$$
and, for $p\neq1/2$,
$$\cov(\Xi(-k),\Xi(-\ell)) \geq 4p(1-p)(2p-1)^{B(k,\ell)} > 0.$$
\end{enumerate}
\end{Prop}
\begin{Remark}
In particular, if $k$ is a power of 2, then $\theta(k)=1$ and
$$\mathbb{E}\big[\Xi(-k)\big] = 2\mathbb{E}\big[\rho(\Xi(0))\big]-1 = (2p-1)^2.$$
\end{Remark}
\begin{proof}
\begin{enumerate}
\item This is a direct consequence of \eqref{eq:pi}.
\item We prove the second statement by repeating the following argument started at $K$.
\begin{eqnarray*}
\lefteqn{\mathbb{P}(\Xi(-K+1)=x_{K-1},\ldots,\Xi(0)=x_0)}\\
& = & \sum_{x_K=\pm1}\prod_{\ell=0}^K\rho\left(\prod_{j=0}^{\ell}x_j^{\beta_{\ell,j}}\right)\ =\ \prod_{\ell=0}^{K-1}\rho\left(\prod_{j=0}^{\ell}x_j^{\beta_{\ell,j}}\right)\sum_{x_K=\pm1}\rho\left(x_K\prod_{j=0}^{K-1}x_j^{\beta_{K,j}}\right),
\end{eqnarray*}
where the last sum takes the form $\rho(y)+\rho(-y)$ which equates to 1.
\item To prove the last two statements, we use Proposition \ref{irreducible} and more specifically the fact that $(\widehat{X}_n)_n$ mixes perfectly after $K+1$ steps so that $\widehat{X}_{K+1}$ has distribution $\widehat{\pi}$. Using \eqref{xirep}, we can write
$$\mathbb{E}\big[\Xi(-k)\big] = \mathbb{E}\big[\widehat{X}_{K+1}(-k)\big] = \mathbb{E}\left[\prod_{j=0}^k\widehat{X}_{K+1-j}(0)^{\beta_{k,j}}\right] = \prod_{j=0}^k\mathbb{E}\left[\widehat{X}_{K+1-j}(0)^{\beta_{k,j}}\right].$$
Now
$$\mathbb{E}\left[\widehat{X}_{K+1-j}(0)^{\beta_{k,j}}\right] = \left\{
\begin{array}{ll}
1 & \beta_{k,j}=0\\
2p-1 & \beta_{k,j}=1
\end{array}\right.$$
Using Lucas Theorem, we know that exactly $2^{\theta(k)}$ $j$'s are such that $\beta_{k,j}=1$. See Lemma \ref{beta} of the Appendix. The result follows.
\item Again using \eqref{xirep} and Lemma \ref{B} of the Appendix, we get
\begin{eqnarray*}
\mathbb{E}\big[\Xi(-k)\Xi(-\ell)\big] & = & \mathbb{E}\big[\widehat{X}_{K+1}(-k)\widehat{X}_{K+1}(-\ell)\big]\\
& = & \mathbb{E}\left[\left(\prod_{j=0}^k\widehat{X}_{K+1-j}(0)^{\beta_{k,j}}\right)\left(\prod_{j=0}^\ell\widehat{X}_{K+1-j}(0)^{\beta_{\ell,j}}\right)\right]\\
& = & \prod_{j=0}^{k\vee\ell}\mathbb{E}\left[\widehat{X}_{K+1-j}(0)^{\beta_{k,j}+\beta_{\ell,j}}\right]\ =\ (2p-1)^{B(k,\ell)}.
\end{eqnarray*}
\end{enumerate}
\end{proof}

We conclude the section with a proof of Theorem \ref{recS} which states that, in terms of recurrence and transience, the downward process $(X_n(1),\ldots,X_n(K))$ behaves like a classical $K$-dimensional random walk. To this end we recall some facts for (finite-state) Markov chains.

For $e\in E$, we write $\tau_m(e)$ for the successive hitting times of state $e$: $\tau_0(e)=0$ and $\tau_m(e)=\inf\{k>\tau_{m-1}(e):X_k=e\}$. Then, $(\tau_m(e)-\tau_{m-1}(e))_m$ is a sequence of independent and identically distributed random variables with finite exponential moments. And more generally for any function $g$, the random variables
$$Y_m = \sum_{k=\tau_{m-1}(e)+1}^{\tau_m(e)}g(X_k)$$
are independent and identically distributed and have (at least) a finite mean. It now follows, by application of the classical Law of Large Numbers for i.i.d. sequences, that
$$\frac1n\sum_{k=1}^{\tau_n(e)}g(X_k) = \frac1n\sum_{m=1}^n\sum_{k=\tau_{m-1}(e)+1}^{\tau_m(e)}g(X_k) = \frac1n\sum_{m=1}^nY_m \xrightarrow{\mbox{\footnotesize{ a.s. }}} \mathbb{E}_e\left[\sum_{k=1}^{\tau_1(e)}g(X_k)\right].$$
Writing $N_{\tau_n(e)}(u)$ for the number of visits to $u$ by time $\tau_n(e)$ (that is the sum of $Y_1,\ldots,Y_n$ when the function $g$ is the indicator function $1_{\{u\}}$), we also get that
$$\frac1n\sum_{k=1}^{\tau_n(e)}g(X_k) = \frac1n\sum_{u\in E}N_{\tau_n(e)}(u)g(u) \xrightarrow{\mbox{\footnotesize{ a.s. }}} \mathbb{E}_e[\tau_1(e)]\sum_{u\in E}\pi(e)g(u) = \mathbb{E}_e[\tau_1(e)]\mathbb{E}[g(\Xi)],$$
where $\Xi$ is a random variable with distribution $\pi$. In total,
\begin{equation}\label{LLN:MC}
\frac1n\sum_{k=1}^{\tau_n(e)}g(X_k)\xrightarrow{\mbox{\footnotesize{ a.s. }}}\mathbb{E}_e\left[\sum_{n=1}^{\tau_1(e)}g(X_n)\right] = \mathbb{E}_e[\tau_1(e)]\mathbb{E}[g(\Xi)].
\end{equation}

\begin{Prop}\label{recg}
Suppose $g$ takes values in $\mathbb{R}^d$. Then the process $Z_n=\sum_{m=1}^ng(X_m)$, with $Z_0=0$, visits any neighbourhood of 0 infinitely often if and only if so does the process $Z_{\tau_n(e)}$.
\end{Prop}
\begin{proof}
The necessity is trivial. To show sufficiency, we first observe that, since the chain takes finitely many values, so do the increments of the process $Z_n$. It follows that $\max_{n\leq2K+1}\|Z_n\|\leq (2K+1)\max_{e\in E}\|g(e)\|$. Call this bound $M$.

Next, we fix a state, say $\mathbf{1}$, and call upon Proposition \ref{irreducible} to show that, uniformly in the initial state, $\tau_1(\mathbf{1})\leq2K+1$ on an event of positive probability.

We reason by contradiction and assume that $Z_n$ visits an $\varepsilon$-ball of 0, $V_\varepsilon$, infinitely often but that $Z_{\tau_n(e)}$ only visits the $2M$-ball of 0, $V_{2M}$, finitely many times. We choose $\varepsilon<M$. Then after the last visit to $V_{2M}$, the process $Z_n$ must return to $V_\varepsilon$, and from there, in no more than $(2K+1)$ steps, and with positive probability, the chain returns to $\mathbf{1}$. At that time, $\tau$, we have the contradicting statements that $\|Z_\tau\|\leq M+\varepsilon$ (this is the farthest that the chain can travel on the aforementioned event of positive probability) and $\|Z_\tau\|>2M$ ($Z_{\tau_n(e)}$ has exhausted its visits to $V_{2M}$).
\end{proof}

%
\subsection*{Proof of Theorem \ref{recS}}
The proof immediately follows from Proposition \ref{recg} and Equation \eqref{LLN:MC}.

\begin{enumerate}
\item $d=1$. Let $g(e)=e(k)$. Then $\mathbb{E}_e[Y_1]=\mathbb{E}_e[\tau_1(e)]\mu(k)$. If $k\in\II[-K,0\II]$, then the biased random walk $Z_{\tau_n(e)}=\sum_{m=1}^nY_m$ is transient (visits a neighbourhood of 0 finitely many times), from which we deduce that $S_n^{\ssup k}=Z_n$ itself is transient. If $k\in\II[1,K\II]$, then $\mathbb{E}[Y_1]=0$ and, by the Chung-Fucks Theorem (see for example \cite{Du10}), the random walk $Z_{\tau_n(e)}$ is recurrent. It follows from Proposition \ref{recg} that $S_n^{\ssup k}$ visits 0 infinitely often.
\item $d=2$. Let $g(e)=(e(k),e(\ell))$, $k,\ell\in\II[1,K\II]$. Here again $\mathbb{E}[Y_1]=0$ and, by the Chung-Fucks Theorem, the 2-dimensional random walk $Z_{\tau_n(e)}$ is recurrent, and, by Proposition \ref{recg}, so is $(S_n^{\ssup k},S_n^{\ssup \ell})$.
\item $d\geq3$. Let $g$ be the projection on a given $d$-dimensional space. Then, the transience of $Z_n$ follows from the transience of the $d$-dimensional random walk $Z_{\tau_n(e)}=\sum_{m=1}^nY_m$.
\end{enumerate}

\section{The invariance principle}

To show that $U_n$ converges weakly to a (correlated) $(2K+1)$-dimensional Brownian motion, we show that for any $a=(a_{-K},\ldots,a_K)\in\mathbb{R}^{2K+1}$, $\sum_{k=-K}^Ka_k\overline{S}^{\ssup k}_n(t)$ converges to a (non-)standard Brownian motion.

The argument is standard, follows from the classical theory of the `stability' of Markov chains, as is for example described in \cite{MT93}, and is briefly described below.

Fix such $a$ and let, for $e\in E$,
$$f(e)=\sum^K_{k=-K}a_ke (k)\quad\mbox{and}\quad\bar{f}(e)=f(e)-\mathbb{E}_\pi[f(X_0)] = \sum_{k=-K}^0a_k\big(e(k)-\mu(k)\big) + \sum_{k=1}^Ka_ke(k).$$
Now, define the stochastic process $\big(S_n(t)\big)_{t\in[0,1]}$ as
$$S_n(t) = \sum_{m=1}^{\lfloor nt \rfloor}\bar{f}(X_m) = \sum_{k=-K}^0a_k\big(S_n^{(k)}(t)-\mu(k)\lfloor nt \rfloor\big) + \sum_{k=1}^Ka_kS_n^{(k)}(t)$$
and consider the linear interpolation
$$\bar{S}_n(t):=S_n(t)+\left(nt-\lfloor nt \rfloor\right)\bar{f}\big(X_{\lfloor nt \rfloor+1}\big).$$
$\bar{S}_n(t)$ satisfies the following functional central limit theorem.

A direct application of Theorem 17.4.4 and Section 17.4.3 of \cite{MT93} yields the following result.
\begin{Theorem}
Let
\begin{equation}
\sigma^2=\var_\pi(f(X_0)) + 2\sum_{n=1}^\infty\cov_\pi\big(f(X_0),f(X_n)\big).\label{v}
\end{equation}
If $\sigma^2>0$, then, for any initial point $e\in E$,
$$\frac{1}{\sigma\sqrt{n}}\bar{S}_n(t)\xrightarrow{\text{ d }}(W_t)_{t\in[0,1]},$$
where $(W_t)_{t\in[0,1]}$ is a standard Brownian motion.
\end{Theorem}

The next step if to check that $\sigma^2>0$. To this end we call upon the following representation of $\sigma^2$ (see Section 17.4.3 of \cite{MT93}). Writing, for simplitcity, $\tau(e)$ for $\tau_1(e)$,
\begin{equation}
\sigma^2 = \pi(e)\mathbb{E}_e\left[\left(\sum_{n=1}^{\tau(e)}\bar{f}(X_n)\right)^2\right].\label{sigma}
\end{equation}

The proof of Proposition \ref{sigmapos} requires the following proposition and notation. For any $e\in E$ and any $k$, we call $e_{\star k}$ the vector whose components coincide with those of $e$ except for component $k$ which is reversed:
$$e_{\star k}(\ell)=\left\{\begin{array}{ll}
e(\ell) & \ell\neq k\\ -e(\ell) & \ell=k
\end{array}\right.$$

\begin{Prop}\label{return1}
Let $e\in E$ be such that for any $k\in\II[-K,0\II]$, $e(k)=1$. Then,
$$\mathbb{P}_e(\tau(e)<+\infty,X_1(0)=\ldots=X_{\tau(e)}(0)=1)>0.$$
\end{Prop}
\begin{proof}
We proceed by induction on $K$. To this end we index the chain itself by $K$ and write $X^{(K)}$ etc. The statement is clearly true for $K=1$ in which case $\mathbb{P}(\tau(e)=1)=p$. Assume it true up to $K-1$. Let $e\in E$ and $e'$ be its projection on components $(-K+1,\ldots,K-1)$. Using the induction assumption, we know that there is a finite number of steps with $X_n^{(K-1)}(k)=1$, for any $k\in\II[-K+1,0\II]$, that take the reduced chain $X^{(K-1)}$ from $e'$ to itself. For the same transitions, the original chain $X^{(K)}$ started from $e$ must end up at either $e$ or $e_{\star K}$. In the latter case, applying the same sequence of transition to $X^{(K)}$ now at $e_{\star K}$ must lead to $e$, which completes the proof.
\end{proof}

\begin{Prop}\label{sigmapos}
For $a\neq0$, $\sigma^2>0$.
\end{Prop}
\begin{proof}
Note that the term on the right hand side of \eqref{sigma} is in fact independent of the choice of the state $e$. We proceed by contradiction. That is, let $f$ be such that $\sigma^2=0$, we show that $f$ must be identically nil. Indeed in this case the random variable $\sum_{n=1}^{\tau_e}\bar{f}(X_n)$ must be $\mathbb{P}_e$ almost surely nil.
Taking $e=\mathbf{1}$, for which $\mathbb{P}_e(\tau(e)=1)=p>0$, yields that $\bar{f}(\mathbf{1})=0$ or equivalently that $\sum_{\ell=-K}^0a_k\big(1-\mu(k)\big)+\sum_{\ell=1}^Ka_k=0$.
Repeating the previous argument with $e=\mathbf{1}_{\star K}$, for which we have again $\mathbb{P}_e(\tau(e)=1)=p>0$, yields that $\sum_{\ell=-K}^0a_k\big(1-\mu(k)\big)+\sum_{\ell=1}^{K-1}a_k-a_K=0$. Combining both equations gives that $a_K$ must be nil.
We proceed using an induction argument. Suppose we have established that $a_\ell=0$ for $\ell\in\{k+1,\ldots,K\}$. Choose $e=\mathbf{1}_{\star k}$. Then Proposition \ref{return1} proves the existence of a finite sequence of transitions that take the chain from $e$ (with $e(k)=1$, for $k\in\II[-K,0\II]$) back to $e$ using ones on the first component. This is represented in Figure \ref{3dot1}.

\begin{figure}[h]
\centering
\includegraphics[width=.75\textwidth]{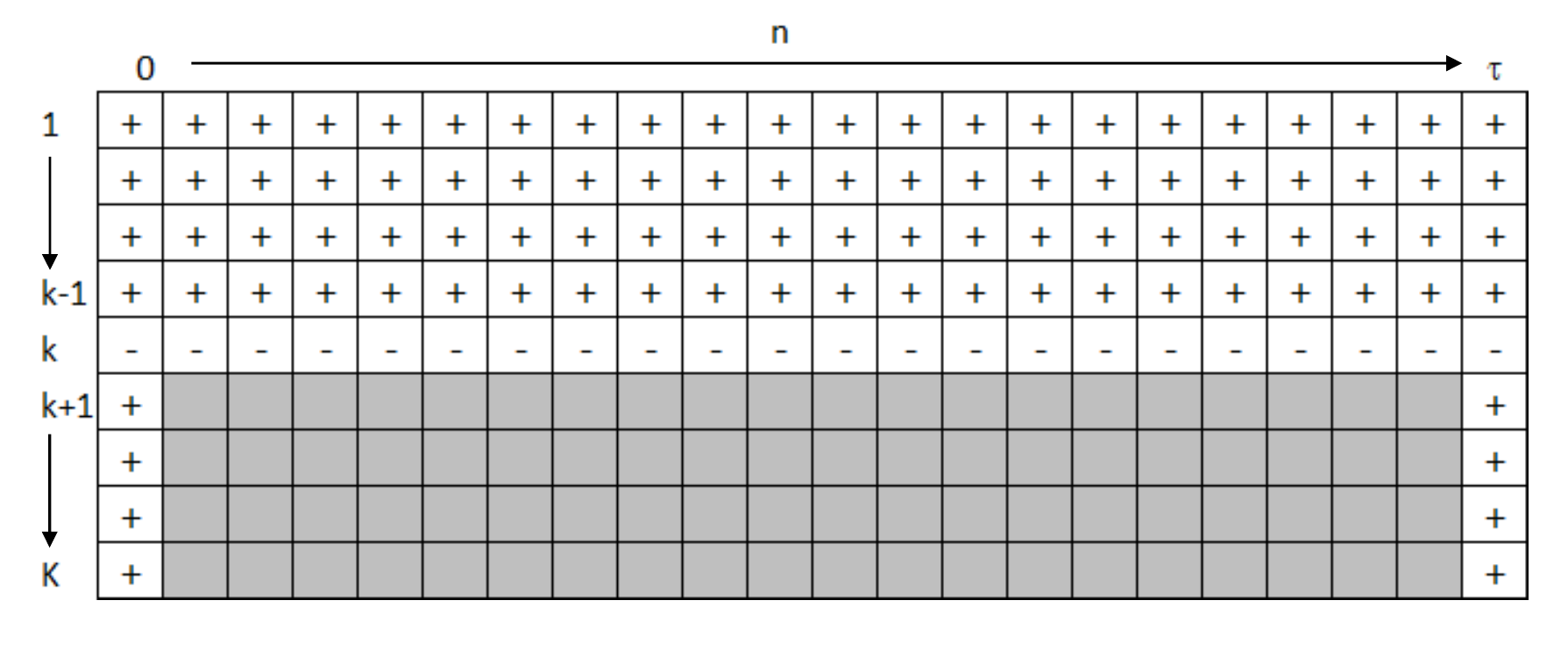}
\centering
\caption{From $\mathbf{1}_{\star k}$ to $\mathbf{1}_{\star k}$ using 1's. ``-'' for $-1$, ``+'' for $1$ and grey cells can be either.}\label{3dot1}
\end{figure}

Writing $\sum_{n=1}^{\tau_e}\bar{f}(X_n)$ is nil on the set of positive probability $\Lambda_e=\{\tau_e<+\infty,X_1(0)=\ldots=X_{\tau_e}(0)=1\}$, yields $\sum_{\ell=-K}^0a_k\big(1-\mu(k)\big)+\sum_{\ell=1}^{k-1}a_\ell-a_k=0$, where we have used the induction assumption that $a_{k+1}=\ldots=a_K=0$ and the fact that $\bar{f}(X_1)=\ldots=\bar{f}(X_{\tau_e})$ on $\Lambda_e$. Combining this equation with that obtained from applying the reasoning to $e=\mathbf{1}$ gives $a_k=0$. This can be repeated for any $k\in\II[0,K\II]$.

The next step is to look at $k\in\II[-K,-1\II]$. To this end we write a variation of \eqref{xirep} that takes row $k$ as reference and proceeds downward:
$$X_n(k+j) = \left(\prod_{\ell=0}^{j-1}X_0(k-\ell+j)^{\alpha_{\ell,n}}\right)\left(\prod_{m=1}^nX_{n-m+1}(k)^{\nu_{j,m}}\right),\quad 1\leq j\leq-k.$$
This enables to show that, with $e=\mathbf{1}_{\star k}$ and $h\in\N$ such that $2^{h-1}\leq-k\leq2^h-1$,
\begin{equation}
\mathbb{P}_e(\widehat{X}_{2^h}=\widehat{e},\widehat{X}_1(k)=-1,\ldots,\widehat{X}_{2^h-1}(k)=-1)>0.\label{xireprev}
\end{equation}
For example, for $k=-1$, taking $\widehat{X}_1(0)=-1$ and $\widehat{X}_2(0)=1$ yields a return to $\widehat{\mathbf{1}_{\star k}}$ after exactly 2 steps such that $\widehat{X}_0(-1)=\widehat{X}_1(-1)=\widehat{X}_2(-1)=-1$.

To show \eqref{xireprev}, all we need to do is show that the $(\widehat{X}_n)_n$ returns to $\widehat{\mathbf{1}_{\star k}}$ after precisely $2^h$ steps and conditional on $\{\widehat{X}_1(k)=-1,\ldots,\widehat{X}_{2^h-1}(k)=-1\}$. Indeed, for $1\leq j\leq-k$,
$$\widehat{X}_{2^h}(k+j) = \left(\prod_{\ell=0}^{j-1}{\underbrace{\widehat{X}_0(k-\ell+j)}_{=1}}^{\alpha_{\ell,2^h}}\right)\left(\prod_{m=1}^{2^h}{\underbrace{\widehat{X}_{n-m+1}(k)}_{=-1}}^{\nu_{j,m}}\right)
= (-1)^{\nu_{j+1,2^h}} = 1,$$
where the last equality follows from the fact, for any $\kappa\in\II[0,2^h\II]$, $\nu_{\kappa,2^h}=0$ (see Proposition 18 of \cite{CHS15}).

As for the case $k\geq0$, we proceed by induction. Recall that $a_k=0$ for any $k\geq0$. Taking $e=\mathbf{1}$ and $e=\mathbf{1}_{\star(-1)}$, we get that $\sum_{\ell=-K}^{-1}a_k\big(1-\mu(k)\big)=0$ and $\sum_{\ell=-K}^{-2}a_k\big(1-\mu(k)\big)+a_{-1}(-1-\mu(-1))=0$, from which we deduce that $a_{-1}=0$. This is then repeated to show that $a_k=0$ for any $k\in\II[-K,K\II]$, proving the result.
\end{proof}

The final step in proving Theorem \ref{TH2} is to untangle the one-dimensional convergence of $\bar{S}_n(t)/(\sigma\sqrt{n})$ into a $(2K+1)$-dimensional convergence of $U_n(t)$. This requires a closer look at, $\sigma^2$, the asymptotic variance of $\bar{S}_n(t)/\sqrt{n}$.

\begin{Prop}\label{variance2}
\begin{enumerate}
\item $\mathbb{E}_{\pi}[f(X_0)]=\sum_{k=0}^Ka_{-k}(2p-1)^{2^{\theta(k)}}$ and
\begin{eqnarray*}
\var_\pi(f(X_0)) & = & \sum_{k=-K}^0\left(1-(2p-1)^{2^{\theta(-k)+1}}\right)a_k^2+\sum_{k=1}^Ka_k^2\\
& &\quad +\ 2\sum_{k=-K+1}^0\sum_{\ell=-K}^{k-1}\left((2p-1)^{B(-k,-\ell)}-(2p-1)^{2^{\theta(-k)}+2^{\theta(-\ell)}}\right)a_ka_\ell.
\end{eqnarray*}
\item Let $n\in\N$ and suppose $2^{\kappa(n)}$ is the largest power of 2 that divides it. Let, for $k,\ell\geq0$ and $n\leq\ell$,
$$\mathcal{B}(k,\ell,n) = 2^{\theta(k)} + 2^{\theta(\ell)} - 2\sum_{j=0}^{k\wedge(-n+\ell)}\beta_{k,j}\beta_{\ell,n+j}.$$
Then,
\begin{eqnarray*}
\lefteqn{\cov_{\pi}(f(X_0),f(X_n))}\\
& = & \sum_{k=-K}^0\sum_{\ell=-K}^{-n}\left((2p-1)^{\mathcal{B}(-k,-\ell,n)}-(2p-1)^{2^{\theta(-k)}+2^{\theta(-\ell)}}\right)a_ka_\ell\\
& &\quad +\ \sum_{k=1}^{K\wedge2^{\kappa(n)}}(2p-1)^{\omega_{k,n}}a_k^2,
\end{eqnarray*}
where the first term on the right hand side (double sum) is nil if $n\geq K+1$.
\item $\bds\sigma^2=\sum_{k,\ell=-K}^K\Sigma(k,\ell)a_ka_\ell\eds$.
\item For any $k\in\II[-K,K\II]$, $\Sigma(k,k)>0$.
\end{enumerate}
\end{Prop}
\begin{proof}
\begin{enumerate}
\item This is a direct consequence of Proposition \ref{pidist}.
\item The result immediately follows from the computation of $\cov_{\pi}(X_0(k),X_n(\ell))$, for any $k,\ell\in\II[-K,K\II]$:
\begin{eqnarray*}
\cov_{\pi}(f(X_0),f(X_n)) & = & \sum_{k,\ell=-K}^Ka_ka_\ell\cov_\pi\big(X_0(k),X_n(\ell)\big)\\
& = & \sum_{k,\ell=-K}^Ka_ka_\ell\Big(\mathbb{E}_{\pi}[X_0(k)X_n(\ell)]-\mu_k\mu_\ell\Big).
\end{eqnarray*}

We start with the case $k\in\II[1,K\II]$ and $\ell\in\II[-K,0\II]$. Using representation \eqref{xirep}, we see that $X_n(\ell)$ is a function of $(X_0(-\ell),\ldots,X_0(0),X_1(0),\ldots,X_n(0))$. Therefore, $X_0(k)$ and $X_n(\ell)$ are independent.

The case $k\in\II[-K,0\II]$ and $\ell\in\II[1,K\II]$ is similar. First, we observe that $X_n(\ell)$ is a function of $(X_0(1),\ldots,X_0(\ell),X_1(0),\ldots,X_n(0))$, and is therefore again independent of $X_0(k)$.

In both of these cases, $\cov_\pi(X_0(k),X_n(\ell))=0$.

Next, we look at the case $k,\ell\in\II[1,K\II]$. Using representation \eqref{xirep} and the independence of $X_0,X_1(0),X_2(0),\ldots$, we write
\begin{eqnarray*}
\mathbb{E}_{\pi}[X_0(k)X_n(\ell)] & = & \mathbb{E}_{\pi}\left[X_0(k)\left(\prod_{j=0}^{\ell-1}X_0(\ell-j)^{\alpha_{j,n}}\right)\left(\prod_{m=1}^nX_{n-m+1}(0)^{\nu_{\ell,m}}\right)\right]\\
& = & \mathbb{E}_{\pi}\left[X_0(k)\left(\prod_{j=0}^{\ell-1}X_0(\ell-j)^{\alpha_{j,n}}\right)\right]\left(\prod_{m=1}^n\mathbb{E}_{\pi}[X_{n-m+1}(0)^{\nu_{\ell,m}}]\right)
\end{eqnarray*}
which, in the case $k<\ell$, yields to $\mathbb{E}_{\pi}[X_0(k)X_n(\ell)]=0$ in view of the fact $\alpha_{0,n}=1$ and, for $\ell\geq1$, $X_0(\ell)$ is symmetric. The case $k>\ell$ is dealt with in exactly the same way yielding again $\mathbb{E}_{\pi}[X_0(k)X_n(\ell)]=0$. Finally, we look at the case $k=\ell\geq1$:
$$\mathbb{E}_{\pi}[X_0(k)X_n(k)] = \left(\prod_{j=1}^{k-1}\mathbb{E}_{\pi}[X_0(k-j)^{\alpha_{j,n}}]\right)\left(\prod_{m=1}^n\mathbb{E}_{\pi}[X_{n-m+1}(0)^{\nu_{k,m}}]\right).$$
Since the $X_0(k-j)$ are symmetric, the first product takes the value 1 if $\alpha_{j,n}=0$ for any $j\in\II[1,k-1\II]$, and takes the value 0 otherwise. Using Lemma \ref{alpha1} of the Appendix, we see that, letting $2^\kappa$ be the largest power of 2 that divides $n$, if $k\leq 2^\kappa$, then the aforesaid product is 1, and if $k>2^\kappa$, then the product is 0. The second product equals $(2p-1)^{\omega_{k,n}}$.

The last case is $k,\ell\in\II[-K,0\II]$. Here, our focus is the upward process $(\widehat{X}_n)_n$, which reaches stationarity after $K+1$ steps (see Proposition \ref{irreducible}):
\begin{eqnarray*}
\lefteqn{\mathbb{E}_{\pi}[X_0(k)X_n(\ell)]\ =\ \mathbb{E}_{\widehat{\pi}}[\widehat{X}_0(k)\widehat{X}_n(\ell)]\ =\ \mathbb{E}[\widehat{X}_{K+1}(k)\widehat{X}_{n+K+1}(\ell)]}\\
& = & \mathbb{E}\left[\left(\prod_{j=0}^{-k}X_{K+1-j}(0)^{\beta_{-k,j}}\right)\left(\prod_{j=0}^{-\ell}X_{n+K+1-j}(0)^{\beta_{-\ell,j}}\right)\right]\\
& = & \mathbb{E}\left[\left(\prod_{m=K+1+k}^{K+1}X_m(0)^{\beta_{-k,K+1-m}}\right)\left(\prod_{m=n+K+1+\ell}^{n+K+1}X_m(0)^{\beta_{-\ell,n+K+1-m}}\right)\right].
\end{eqnarray*}
Clearly, in the case $n+\ell\geq1$, the two products are independent and we have once more $\cov_\pi(X_0(k),X_n(\ell))=0$. If $n\leq-\ell$, then we can rearrange the above products in the 3 independent products the means of which are: $\prod_{j\in I_1\setminus I_2}(2p-1)^{\beta_{-k,j}}$, $\prod_{j\in I_1\cap I_2}(2p-1)^{\beta_{-k,j}+\beta_{-\ell,n+j}-2\beta_{-k,j}\beta_{-\ell,n+j}}$ and $\prod_{j\in I_2\setminus I_1}(2p-1)^{\beta_{-\ell,n+j}}$,
where $I_1=\II[0,-k\II]$ and $I_2=\II[-n,-n-\ell\II]$. However, for $-k\leq-n-\ell$,
\begin{eqnarray*}
\lefteqn{\sum_{j\in I_1\setminus I_2}\beta_{-k,j} + \sum_{j\in I_1\cap I_2}\big(\beta_{-k,j}+\beta_{-\ell,n+j}-2\beta_{-k,j}\beta_{-\ell,n+j}\big) + \sum_{j\in I_2\setminus I_1}\beta_{-\ell,n+j}}\\
& = & \sum_{j=0}^{-k}\beta_{-k,j} + \sum_{j=0}^{-k}\beta_{-\ell,n+j} - 2\sum_{j=0}^{-k}\beta_{-k,j}\beta_{-\ell,n+j} + \sum_{j=-n}^{-1}\beta_{-\ell,n+j} + \sum_{j=-k+1}^{-n-\ell}\beta_{-\ell,n+j}\\
& = & 2^{\theta(-k)} + 2^{\theta(-\ell)} - 2\sum_{j=0}^{-k}\beta_{-k,j}\beta_{-\ell,n+j}.
\end{eqnarray*}
Similarly, if $-k>-n-\ell$,
\begin{eqnarray*}
\lefteqn{\sum_{j\in I_1\setminus I_2}\beta_{-k,j} + \sum_{j\in I_1\cap I_2}\big(\beta_{-k,j}+\beta_{-\ell,n+j}-2\beta_{-k,j}\beta_{-\ell,n+j}\big) + \sum_{j\in I_2\setminus I_1}\beta_{-\ell,n+j}}\\
& = & \sum_{j=-n-\ell+1}^{-k}\beta_{-k,j} + \sum_{j=0}^{-n-\ell}\beta_{-k,j} + \sum_{j=0}^{-n-\ell}\beta_{-\ell,n+j} - 2\sum_{j=0}^{-n-\ell}\beta_{-k,j}\beta_{-\ell,n+j} + \sum_{j=-n}^{-1}\beta_{-\ell,n+j}\\
& = & 2^{\theta(-k)} + 2^{\theta(-\ell)} - 2\sum_{j=0}^{-n-\ell}\beta_{-k,j}\beta_{-\ell,n+j}.
\end{eqnarray*}

We immediately get that
$$\mathbb{E}_{\pi}[X_0(k)X_n(\ell)] = (2p-1)^{\mathcal{B}(-k,-\ell,n)}.$$

\item Theorem 17.0.1 of \cite{MT93} states that
$$\sigma^2 = \var_\pi(f(X_0))+2\sum_{n=1}^\infty\cov_{\pi}(f(X_0),f(X_n)).$$
Substituting the expressions from 1. and 2. leads to the required representation.

\item Finally, the positivity of $\Sigma(k,k)$ is a direct consequence of Proposition \ref{sigmapos}. Simply take $a_k=1$ and $a_\ell=0$ for $\ell\neq k$.
\end{enumerate}
\end{proof}

\section{Appendix}

\begin{lemma}\label{beta}
Let, for $k=\sum_{i\geq0}k_i2^i$, $\theta(k)=\sum_{i\geq0}k_i$ be the number of the base 2 digits of $k$ that are equal to 1. Then,
$$\#\{j\in\II[0,k\II]:\beta_{k,j}=1\} = 2^{\theta(k)}.$$
\end{lemma}
\begin{proof}
By Lucas Theorem, $\beta_{k,j}=1$ if and only if, for every $i\geq0$, $k_i\geq j_i$. This leads to $j_i=0$ whenever $k_i=0$. On the other hand, if $k_i=1$, then $j_i$ can be either 0 or 1. The result immediately follows.
\end{proof}

\begin{lemma}\label{B}
Let, for $k=\sum_{i\geq0}k_i2^i$ and $\ell=\sum_{i\geq0}\ell_i2^i$, $\left<k,\ell\right>=\sum_{i\geq0}k_i\ell_i2^i$. Then,
$$B(k,\ell) = 2^{\theta(k)}+2^{\theta(\ell)}-2\sum_{j=0}^{k\wedge\ell}\beta_{k,j}\beta_{\ell,j} = 2^{\theta(k)}+2^{\theta(\ell)}-2^{\theta(\left<k,\ell\right>)+1}.$$
$\theta(\left<k,\ell\right>)$ is the number of the base 2 digits of $k$ and $\ell$ that are concurrently equal to 1. In particular, $B(k,k)=0$.
\end{lemma}
\begin{proof}
Since $\beta_{k,j}+\beta_{\ell,j}=1$ if and only if $\beta_{k,j}+\beta_{\ell,j}-2\beta_{k,j}\beta_{\ell,j}=1$ and the second expression is always in $\{0,1\}$,
\begin{eqnarray*}
B(k,\ell) & = & \#\{j\in\II[1,k\vee\ell\II]:\beta_{k,j}+\beta_{\ell,j}=1\}\ =\ \sum_{j=1}^{k\vee\ell}\left(\beta_{k,j}+\beta_{\ell,j}-2\beta_{k,j}\beta_{\ell,j}\right)\\
& = & \sum_{j=1}^k\beta_{k,j}+\sum_{j=1}^\ell\beta_{\ell,j}-2\sum_{j=1}^{k\wedge\ell}\beta_{k,j}\beta_{\ell,j}\\
& = & \big(2^{\theta(k)}-1\big)+\big(2^{\theta(k)}-1\big)-2\big(2^{\theta(\left<k,\ell\right>)}-1\big).
\end{eqnarray*}
In the last step, we use the fact that $\beta_{k,j}=\beta_{\ell,j}=1$ if and only if $\beta_{\left<k,\ell\right>,j}=1$.
\end{proof}

\begin{lemma}\label{alpha1}
If $2^\kappa$ is the largest power of 2 that divides $n$ (i.e. $n=m2^\kappa$ where $m$ is odd), then $\min\{j\geq1:\ \alpha_{j,n}=1\}=2^\kappa$.
\end{lemma}
\begin{proof}
Let $n=\sum_{\ell\geq0}n_\ell2^\ell$ be the binary representation of $n$. If $n=m2^\kappa$, then $n_\ell=0$ for $\ell<\kappa$ and $n_\kappa=1$, and
$$n-1=\sum_{\ell=0}^{\kappa-1}2^\ell+\sum_{\ell\geq\kappa+1}n_\ell2^\ell.$$
Let $j\in\II[1,2^\kappa-1\II]$, $j=\sum_{\ell=0}^{\kappa-1}j_\ell2^\ell$ be its binary representation and $\ell^\star=\max\{\ell\in\II[0,\kappa-1\II]:\ j_\ell=1\}$. Then the $\ell^\star$ digit of $n+j-1$ is 0 while the same digit of $j$ is 1. By Lucas Theorem, $\alpha_{j,n}=0$. On the other hand, all digits of $n+2^\kappa-1$ are at least as large as the same digits of $2^\kappa$. Again, by Lucas Theorem, $\alpha_{2^\kappa,n}=1$.
\end{proof}

\end{document}